\documentclass[a4paper]{amsart}
\usepackage[utf8]{inputenc}
\usepackage[T1,T2A]{fontenc}
\usepackage[french,russian,english]{babel}
\usepackage[unicode]{hyperref}
\usepackage{amsmath,amssymb,amsfonts,amsthm}
\usepackage{graphicx}
\usepackage[backgroundcolor=white]{todonotes}

\usepackage{tikz}
\usepackage[backend=biber,firstinits,autolang=langname]{biblatex}
\usetikzlibrary{arrows}
\expandafter\addto\csname HyLang@english\endcsname{

}
\newtheorem{theorem}{Theorem}
\newtheorem*{theorem*}{Theorem}
\newtheorem*{maintheoremThm}{Main Theorem}
\newenvironment{maintheorem}{%
\begin{maintheoremThm}%
\phantomsection
\expandafter\def\csname @currentlabelname\endcsname{Main Theorem}
}{%
\end{maintheoremThm}%
}
\newtheorem{lemma}{Lemma}
\newtheorem{proposition}{Proposition}
\newtheorem*{conjecture*}{Conjecture}

\theoremstyle{remark}
\newtheorem{remark}{Remark}
\newtheorem*{remark*}{Remark}

\theoremstyle{definition}
\newtheorem*{definition*}{Definition}
\newtheorem{question}{Question}

\renewcommand{\Im}{\mathop{\mathrm{Im}}}
\newcommand{\eps}{\varepsilon}
\newcommand{\bbC}{\mathbb C}
\newcommand{\bbR}{\mathbb R}

\newcommand{\bbH}{\mathbb H}

\newcommand{\bbZ}{\mathbb Z}
\newcommand{\rot}{\mathrm {rot}\,}
\DeclareMathOperator{\id}{id}

\title{Complex rotation numbers: bubbles and their intersections}

\author{Nataliya Goncharuk}
\email{ng432@cornell.edu}
\address{%
Cornell University\\
Department of Mathematics\\
310 Malott Hall\\
Ithaca, NY 14853-4201 USA}
\thanks{Supported by RFBR project 16-01-00748-a and Laboratory Poncelet.}
\thanks{Cornell University, Department of Mathematics}
\makeatletter
\hypersetup{
  pdftitle=\@title,
  pdfauthor=\authors,
  pdfkeywords=\@keywords,
  pdfsubject=\@subjclass
}
\makeatother

\begin{document}

\begin{abstract}
The construction of complex rotation numbers, due to V.Arnold, gives rise to a fractal-like set ``bubbles'' related to a circle diffeomorphism. ``Bubbles'' is a complex analogue to Arnold tongues.

This article contains a survey of the known properties of bubbles, as well as a variety of open questions. In particular, we show that bubbles can intersect and self-intersect, and provide approximate pictures of bubbles for perturbations of Möbius circle diffeomorphisms.

\end{abstract}
\maketitle

\section{Introduction}

\subsection{Complex rotation numbers. Arnold's construction}
In what follows,  $f\colon \bbR/\bbZ \to \bbR/\bbZ$ is an analytic orientation-preserving circle diffeomorphism. Its analytic extension to a small neighborhood of $\bbR/\bbZ$ in $\bbC/\bbZ$ is still denoted by $f$. $\bbH\subset \bbC$ is the open upper half-plane.

The following construction was suggested by V. Arnold  \cite[Sec. 27]{Arn-english} in 1978. Given $\omega\in \bbH/\bbZ$ and a small positive $\eps \in\bbR$, one can construct a complex torus $E(f+\omega)$ as the quotient space of a cylinder $\Pi$ by the action of $f+\omega$:
\begin{gather}
    \label{eq-strip}
    \Pi:=\{z \in \bbC/\bbZ \mid -\eps <\Im z < \Im\omega+\eps\},\\
    E(f+{\omega}) := \Pi /(z\sim f(z)+{\omega}).\notag
\end{gather}
For a small positive $\eps$, the quotient space $E(f+{\omega})$ is a torus, inherits a complex structure from $\bbC/\bbZ$ and does not depend on $\eps$.

Due to the Uniformization theorem, for a unique  $\tau\in \bbH/\bbZ$  there exists a biholomorphism 
\begin{equation}
\label{eq-Hw}
 H_{\omega} \colon E(f+{\omega}) \to \bbC / (\bbZ+\tau \bbZ)
\end{equation}
such that $H_{\omega}$ takes  $\bbR/\bbZ\subset E(f+\omega)$ to a curve homotopic to $\bbR/\bbZ\subset \bbC / (\bbZ+\tau \bbZ)$. The number $\tau({f+\omega}):=\tau \in \bbH/\bbZ$, i.e.\ the modulus of the complex torus $E(f+{\omega})$, is called the \emph{complex rotation number} of $f+{\omega}$.
The term is due to E. Risler, \cite{Ris}.

Complex rotation number $\tau(f+\omega)$ depends holomorphically on $\omega\in \bbH/\bbZ$, see \cite[Sec. 2.1, Proposition 2]{Ris}.

\subsection{Rotation number and its properties}
This section lists well-known results on rotation numbers; see \cite[Sec. 3.11, 3.12]{KaHa} for more details.

Let $f$ be an orientation-preserving circle homeomorphism, and let $F\colon \bbR \to \bbR$ be its lift to the real line. The limit
$$\rot f = \lim_{n\to \infty} \frac{F^{\circ n}(x)}{n} \mod 1$$
exists and does not depend on $x\in \bbR$. It is called the \emph{rotation number} of the circle homeomorphism $f$.

Rotation number is invariant under continuous conjugations of $f$. It is rational, $\rot f = \frac pq$, if and only if $f $ has a periodic orbit of period $q$. If $\rot f$ is irrational and $f\in C^2 (\bbR/\bbZ)$, then $f$ is continuously conjugate to  $z\mapsto z+\rot f$ (Denjoy Theorem, see \cite[Sec 3.12.1]{KaHa}). We will need the following, much more complicated result.
\begin{definition*}
 A real number $\rho$ is called \emph{Diophantine} if there exist $C>0$ and $\beta>0$ such that for all rationals $p/q$,
 $$
 \left|\rho-\frac{p}{q}\right|\ge \frac{C}{q^{2+\beta}}.
 $$
\end{definition*}
\begin{theorem} [M. R. Herman \cite{Herm}, J.-C. Yoccoz \cite{Yoc}]
\label{th-HermYoc}
 If an analytic circle diffeomorphism has a Diophantine rotation number $\rot f$, then it is  analytically conjugate to  $z\mapsto z+\rot f$.
\end{theorem}
This motivates the term ``complex rotation number'' for $\tau(f+\omega)$ above: while a circle diffeomorphism $f$ is conjugate to the rotation $x \mapsto x+\rot f$ on $\bbR/\bbZ$, a complex-valued map $f+\omega$ is biholomorphically conjugate to the complex shift $z \mapsto z+\tau(f+\omega)$ in the cylinder $\Pi \subset \bbC/\bbZ$.
 
\subsection{Steps on the graph of \texorpdfstring{$\omega\mapsto \rot(f+\omega)$}{ω↦rot(f+ω)}.}

Rotation number depends continuously on $f$ in $C^0$-topology. In particular, $\rot(f+{\omega})$ depends continuously on $\omega\in \bbR/\bbZ$; clearly, it (non-strictly) increases on $\omega$.

Recall that a periodic  orbit of a circle diffeomorphism is  called \emph{parabolic} if its multiplier is one, and \emph{hyperbolic} otherwise. If a circle diffeomorphism has  periodic orbits, and they are all hyperbolic, then the diffeomorphism is called \emph{hyperbolic}.

Let $I_{\frac pq}:=\{\omega\in \bbR/\bbZ \mid \rot(f+{\omega})=\frac pq\}$; from now on, we always assume that $p,q$ are coprime. If  for some value of $\omega$, $f+\omega$ has the rotation number $p/q$ and a \emph{hyperbolic} orbit of period $q$, then this orbit persists under a small perturbation of $\omega$. In this case, $I_{p/q}$ is a segment of a non-zero length. Endpoints of $I_{p/q}$ correspond to  diffeomorphisms $f+\omega$ having only parabolic orbits.

In a generic case, the graph of the function $\omega\mapsto \rot (f+\omega)$ contains infinitely many steps, i.e.\ non-trivial segments $I_{\frac pq}\times \{\frac pq\}$, on rational heights.

\subsection{Rotation numbers as boundary values of a holomorphic function}

\begin{question}
    Can we find a holomorphic self-map $\tau$ on $\bbH/\bbZ$ such that its boundary values on $\bbR/\bbZ$ coincide with $\omega\mapsto \rho(f+\omega)$?
\end{question}
 
The answer is No (except for the trivial case $f(x)=x+c$), because the function  $\omega\mapsto \rho(f+\omega)$  is locally constant on non-empty intervals   $I_{\frac pq}$, and this is not possible for boundary values of holomorphic functions. In more detail, note that $\bbH/\bbZ$ is biholomorphically equivalent to the punctured unit disc $D\setminus \{0\}$, so  the map $\frac{1}{2\pi i}\ln z \colon  D\setminus \{0\}\to \bbH/\bbZ$ conjugates $\tau$ to a holomorphic bounded self-map of the punctured unit disc. Clearly, $0$ is a removable singularity for this self-map. The following Luzin--Privalov theorem (see \cite[Section 14, p. 159]{LP}) shows that such extension $\tau$ does not exist:
\begin{theorem}[N. Luzin, J. Privalov]
\label{th-lp}
 If a holomorphic function in the unit disc $D$ has finite non-tangential limits at all points of  $E \subset \partial D$, where $E$ has a non-zero Lebesgue measure, then this function is uniquely defined by these limits.
\end{theorem}

This motivates the next question:
\begin{question}
    Can we find a holomorphic self-map on $\bbH/\bbZ$ such that its boundary values  on $(\bbR/\bbZ)\setminus \bigcup I_{p/q}$ coincide with  $\omega\mapsto \rho(f+\omega)$?
\end{question}
\begin{remark*}
The set $(\bbR/\bbZ)\setminus \bigcup I_{p/q}$ has non-zero measure due to the result of M.~R.~Herman, see \cite[Sec. 6, p. 287]{H77}; so by \autoref{th-lp}, such holomorphic extension must be unique.
\end{remark*}

The answer to this question is Yes, and this holomorphic function is the complex rotation number $\tau(f+\omega)$.  The following theorem is proved in \cite{XB_NG}; the proof is based on previous results by E. Risler,  V. Moldavskij, Yu.~Ilyashenko, and N. Goncharuk \cite{Ris,M-en,YuIM,NG2012-en}).
\begin{theorem}[X. Buff, N. Goncharuk \cite{XB_NG}]
\label{th-XB_NG}
Let $f\colon \bbR/\bbZ \to \bbR/\bbZ$ be an orientation-preserving analytic circle diffeomorphism.  Then the holomorphic function ${\tau(f+\cdot)} \colon \bbH/\bbZ \to \bbH/\bbZ$ has a continuous extension ${\bar \tau(f+\cdot)} \colon \overline{\bbH/\bbZ} \to \overline{\bbH/\bbZ}$. Assume $\omega \in \bbR/\bbZ$.
\begin{itemize}
 \item If $\rot (f+{\omega})$ is irrational, then $\bar \tau(f+\omega)=\rot(f+{\omega})$.
 \item If $\rot (f+{\omega})$ is rational and $f+{\omega}$ has a parabolic periodic orbit, then $\bar \tau(f+\omega)=\rot (f+{\omega})$.
 \item If $\rot (f+{\omega})$ is rational and $f+{\omega}$ is hyperbolic on an open interval  $\omega\in I\subset \bbR/\bbZ$, then $\bar \tau(f+\omega)$ depends analytically on $\omega\in I$ and $\bar \tau(f+\omega)\in \bbH/\bbZ$ for $\omega \in I$.
\end{itemize}
\end{theorem}  
The extension $\bar \tau(f+{\omega})$ is also called the complex rotation number of $f+\omega$. Due to \autoref{th-XB_NG}, it is continuous on $\omega$, and coincides with the ordinary rotation number on $\bbR/\bbZ \setminus \bigcup I_{p/q}$.

  \begin{definition*}
  The image of the segment $I_{\frac pq}=\{\omega\in \bbR/\bbZ \mid \rot(f+{\omega})=\frac pq\}$ under the map $\omega \mapsto \bar \tau(f+\omega)$ is called the \emph{$\frac pq$-bubble} of $f$.
 \end{definition*}

Due to \autoref{th-XB_NG}, the $\frac pq$-bubble is a union of several analytic curves in the upper half-plane with endpoints at $\frac pq$. Each analytic curve corresponds to the interval of hyperbolicity of $f+{\omega}$, and its endpoints correspond to $f+\omega$ with parabolic orbits.

So, each circle diffeomorphism $f$ gives rise to a ``fractal-like'' set $\bar \tau(f+\omega)$ (bubbles) in the upper half-plane, containing countably many analytic curves. The possible shapes of bubbles are not known. The following question is also open.

\begin{question}
 Is the set $\bar \tau(f+\omega)$ self-similar (i.e. is it a fractal set)?
\end{question}

\subsection{Properties of bubbles and \nameref{thm-main}}

\begin{question}
    Is $\bar \tau$ invariant under analytic conjugacies?
\end{question}
The answer is Yes:
\begin{lemma}
\label{lem-invar}
Complex rotation number $\bar \tau$ is invariant under analytic conjugacies: for two analytically conjugate circle diffeomorphisms $f_1,f_2$, we have $\bar \tau(f_1) = \bar \tau(f_2)$.
\end{lemma}
For non-hyperbolic $f_1,f_2$, their complex rotation numbers coincide with rotation numbers, so this lemma trivially repeats the invariance of rotation numbers under conjugacies. For hyperbolic diffeomorphisms, the proof of this lemma is implicitly contained in \cite{XB_NG}, see also \autoref{sec-expl-constr} below.

Note that in general, for conjugate $f_1,f_2$ and $\omega\in \overline{\bbH}/\bbZ$, the numbers $\overline \tau(f_1+\omega)$ and $\overline \tau(f_2+\omega)$ do not coincide.

\begin{question}
    Is there an explicit formula for  $\overline \tau(f+\omega)$?
\end{question}
The only case when the author can obtain an explicit formula for  $\overline \tau(f+\omega)$ is described in the following proposition.

 Let $\pi \colon \bbC/\bbZ \to \bbC^*$ be given by $\pi(z) := \exp(2\pi i z)$.
\begin{proposition}
\label{prop-frac}
 Let $F$ be a Möbius map that preserves the circle $\{|w|=1 \mid w\in \bbC\}$. Let $f\colon \bbR/\bbZ \to \bbR/\bbZ$ be given by $f:= \pi^{-1}\circ F\circ \pi$. Then $f$ has only a $0$-bubble, and this bubble is a vertical segment. 
\end{proposition}
\begin{proof}
First, let us compute $\tau(f+\omega)$ for $\omega \in \bbH/\bbZ$.

Put $F_{\omega}:=e^{2\pi i \omega} F$. For $\omega\in \bbH/\bbZ$ and small $\eps>0$, let $E^*(F_\omega)$ be the quotient space of the annulus $\Pi^*:= \{1>w>|e^{2\pi i (\omega+\eps)}|\}$ via the map $F_{\omega}$. Note that the map $\pi$ induces a biholomorphism of $E(f+\omega)$ to $E^*(F_\omega)$. Indeed, it  takes $\Pi$ to $\Pi^*$ and conjugates $f+\omega$ to $F_{\omega}=\pi\circ(f+\omega)\circ \pi^{-1}$. So $\tau(f+\omega)$ is equal to the modulus of $E^*(F_\omega)$.

The map $F_\omega$ is a Möbius map that takes the unit circle to the interior of the unit disc. Let $A_{\omega}$ be its attractor with multiplier $\mu(\omega)$ and $R_{\omega}$ be its repellor. The map $\frac{w-A_{\omega}}{w-R_{\omega}}$ conjugates $F_{\omega}$ to the linear map $w\mapsto \mu(\omega) w$, thus induces a biholomorphism of $E^*(F_\omega)$ to  the complex torus $\bbC^* / (w\sim \mu(\omega)w)$. The modulus of this torus is equal to $\frac{1}{2\pi i }\ln \mu(\omega)$.
Finally, $\tau(f+\omega)=\frac{1}{2\pi i }\ln \mu(\omega)$.

Now let us study the boundary values of $\tau(f+\omega)$, i.e. $\overline \tau(f+\omega_0) = \lim_{\omega\to \omega_0} \tau(f+\omega)$ for $\omega_0\in \bbR/\bbZ$.

The map $F_{\omega_0}$ is a Möbius self-map of the unit circle. If it has two hyperbolic fixed points on the unit circle (i.e. $\omega_0$ is an interior point of  $I_0$), then the multiplier of its attractor, $\mu(\omega_0)$, is real because $F_{\omega_0}$ preserves the unit circle. Then $\overline \tau(f+\omega_0)  = \lim_{\omega\to \omega_0} \frac{1}{2\pi i }\ln \mu(\omega) = \frac{1}{2\pi i }\ln \mu(\omega_0) \in i\bbR$. If $F_{\omega_0}$ has one parabolic fixed point on the unit circle, then $\lim_{\omega\to \omega_0} \mu(\omega) = 1$, and $\overline \tau(f+\omega_0) =0$.
If $F_{\omega_0}$ has no fixed points on the unit circle (i.e. $\omega_0\in (\bbR/\bbZ) \setminus I_0$), then it has a unique fixed point $A_{\omega_0}$ inside the unit disc and a unique fixed point $R_{\omega_0}$ outside it; the Schwarz lemma implies that the multiplier of $A_{\omega_0}$ satisfies $|\mu(\omega_0)|=1$, so $\overline \tau(f+\omega_0) = \lim_{\omega\to \omega_0} \frac{1}{2\pi i }\ln \mu(\omega) = \frac{1}{2\pi i }\ln \mu(\omega_0) \in \bbR/\bbZ$.

Finally, the image of $I_0$ under $\overline \tau(f+\cdot)$ belongs to $i\bbR$, and the image of $(\bbR/\bbZ) \setminus I_0$ belongs to $\bbR/\bbZ$. We conclude that the only bubble of $f$ is a $0$-bubble, and it is a vertical segment.
\end{proof}

\begin{question}
    Is there a way to compute $\overline{\tau}(f+\omega)$ approximately?
\end{question}

\graphicspath{{bubbles-pic/}}
\begin{figure}[h]
 \begin{center}
  \includegraphics[width=0.5\textwidth]{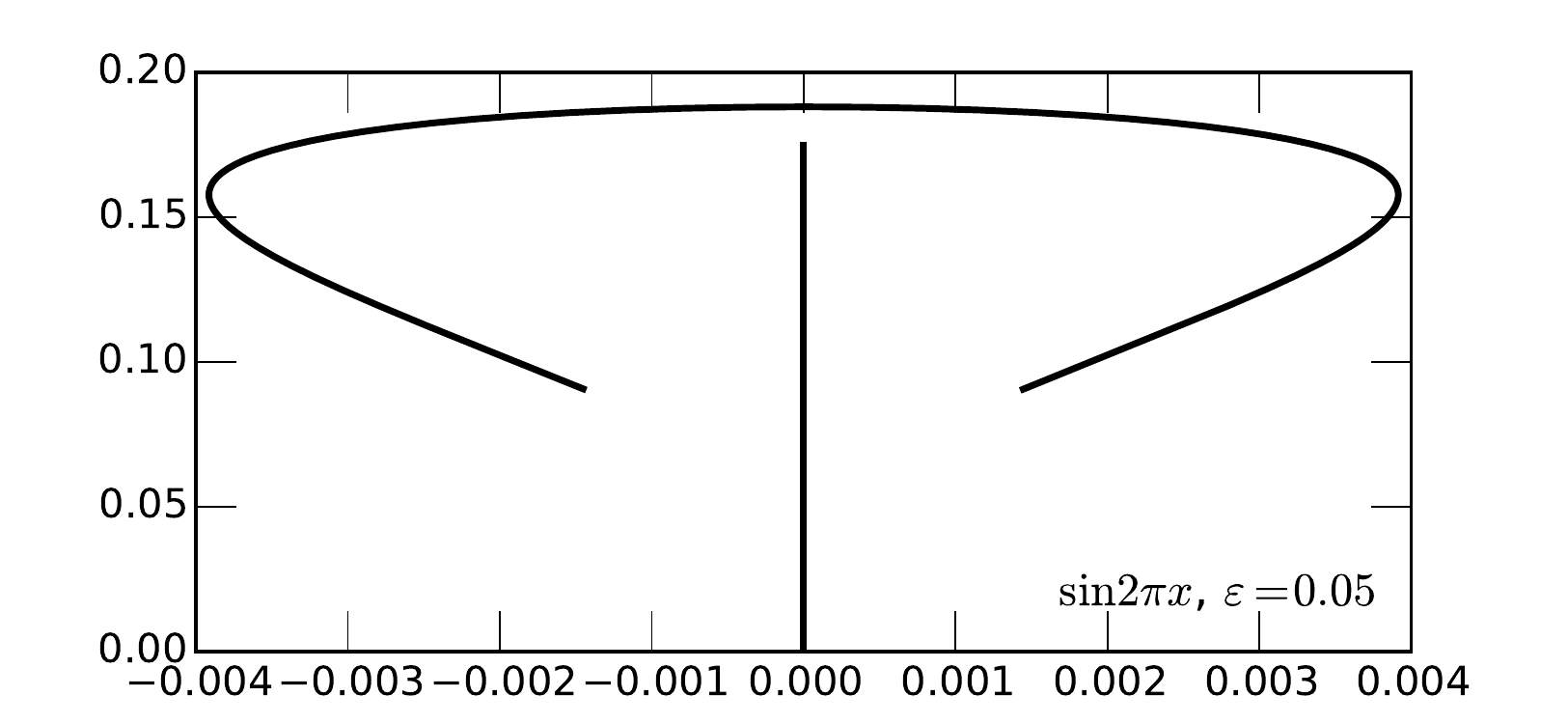}\hfil \includegraphics[width=0.5\textwidth]{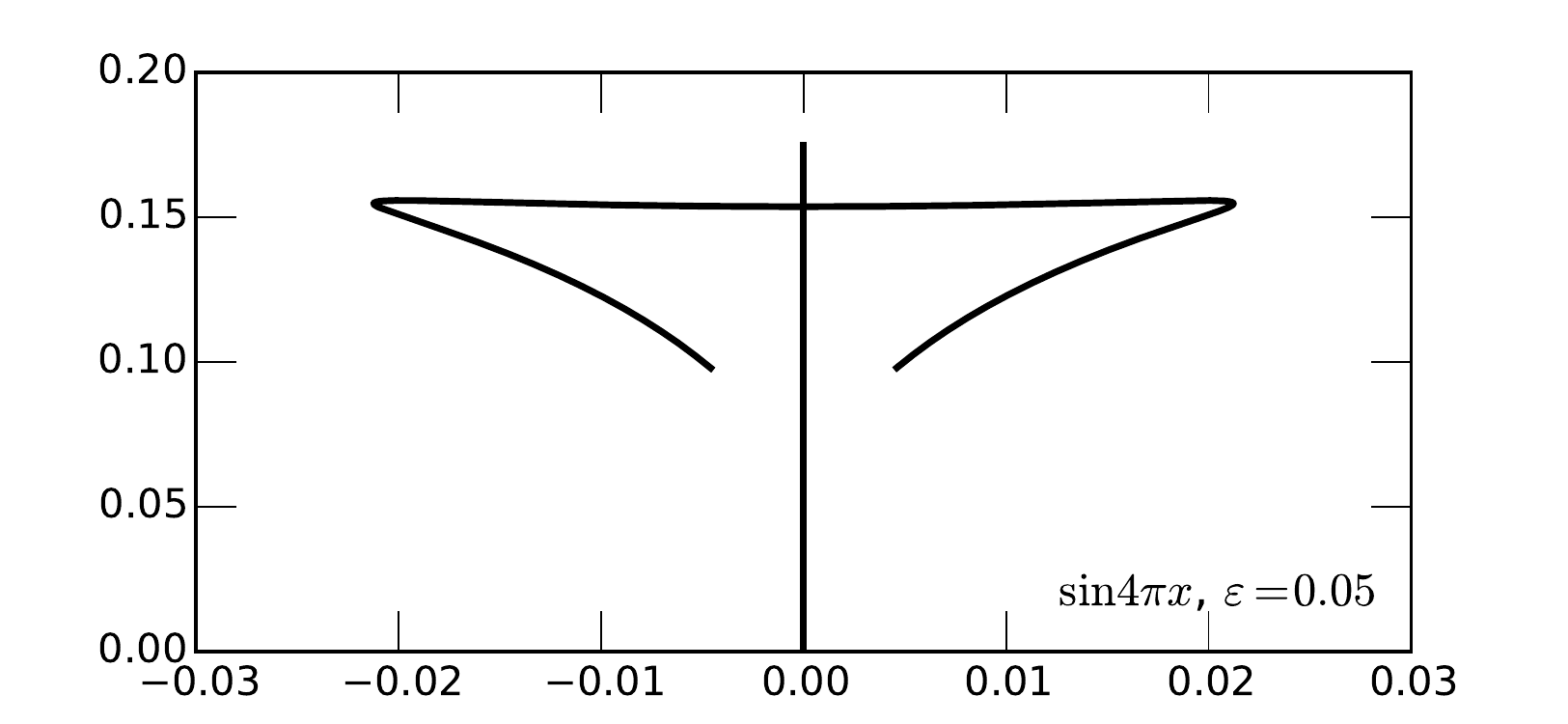}

  \includegraphics[width=0.5\textwidth]{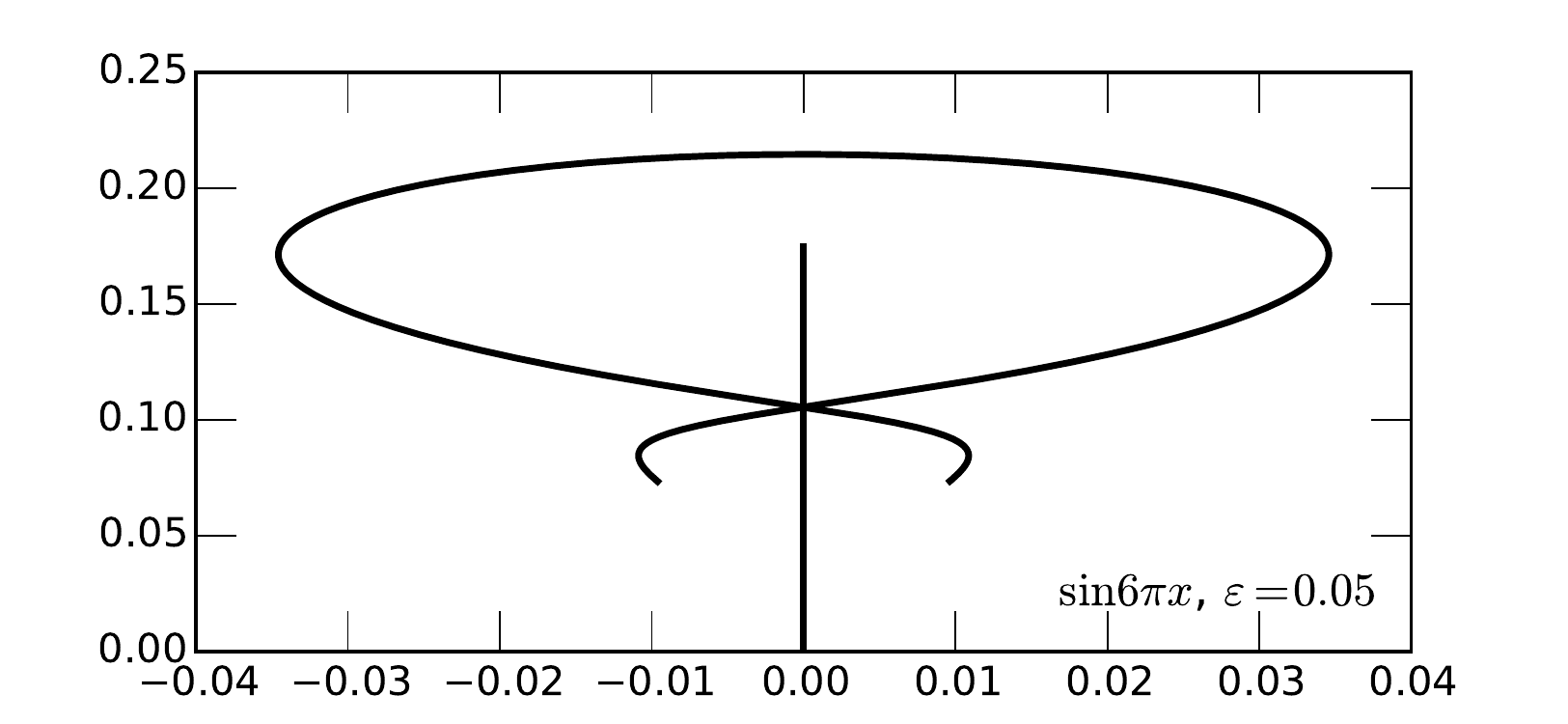}\hfil  \includegraphics[width=0.5\textwidth]{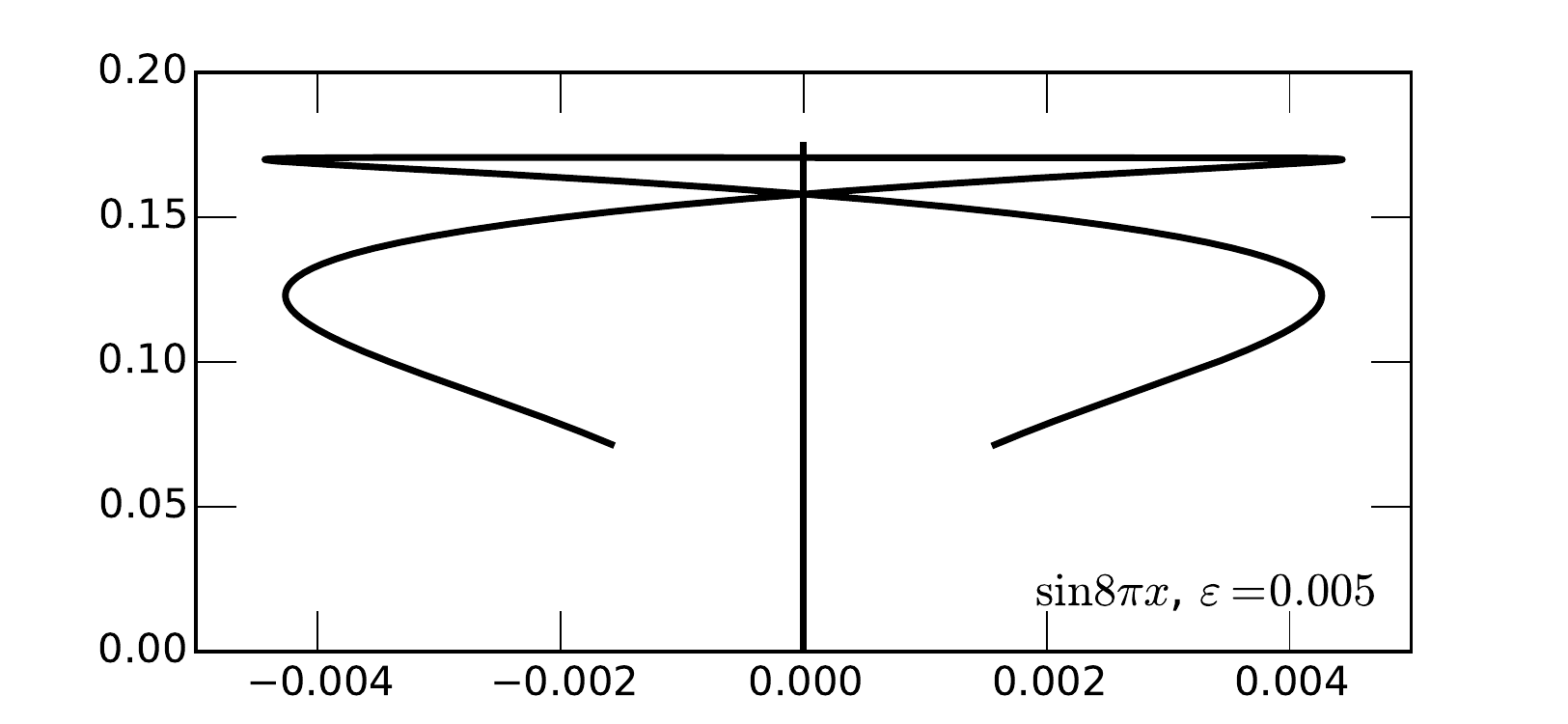}

  \includegraphics[width=0.5\textwidth]{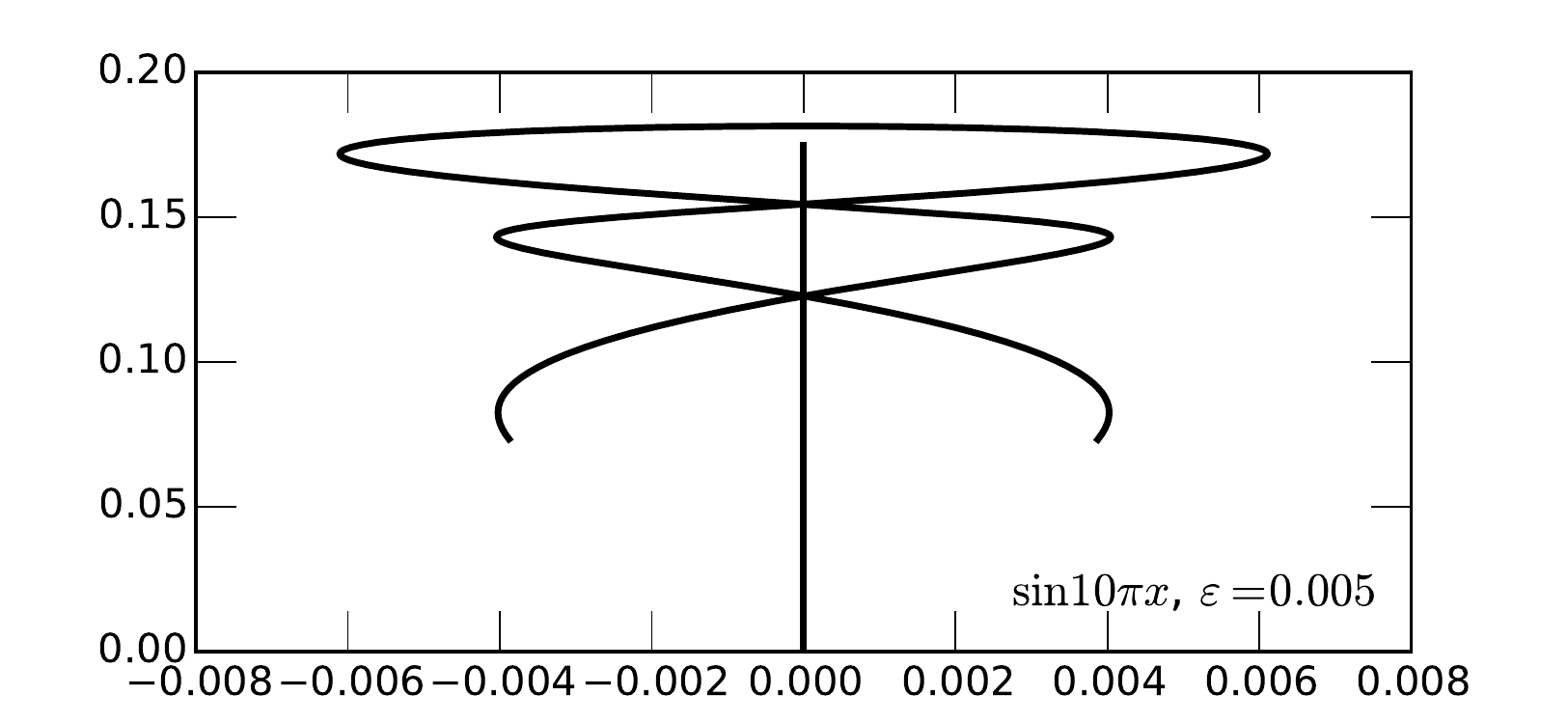}\hfil \includegraphics[width=0.5\textwidth]{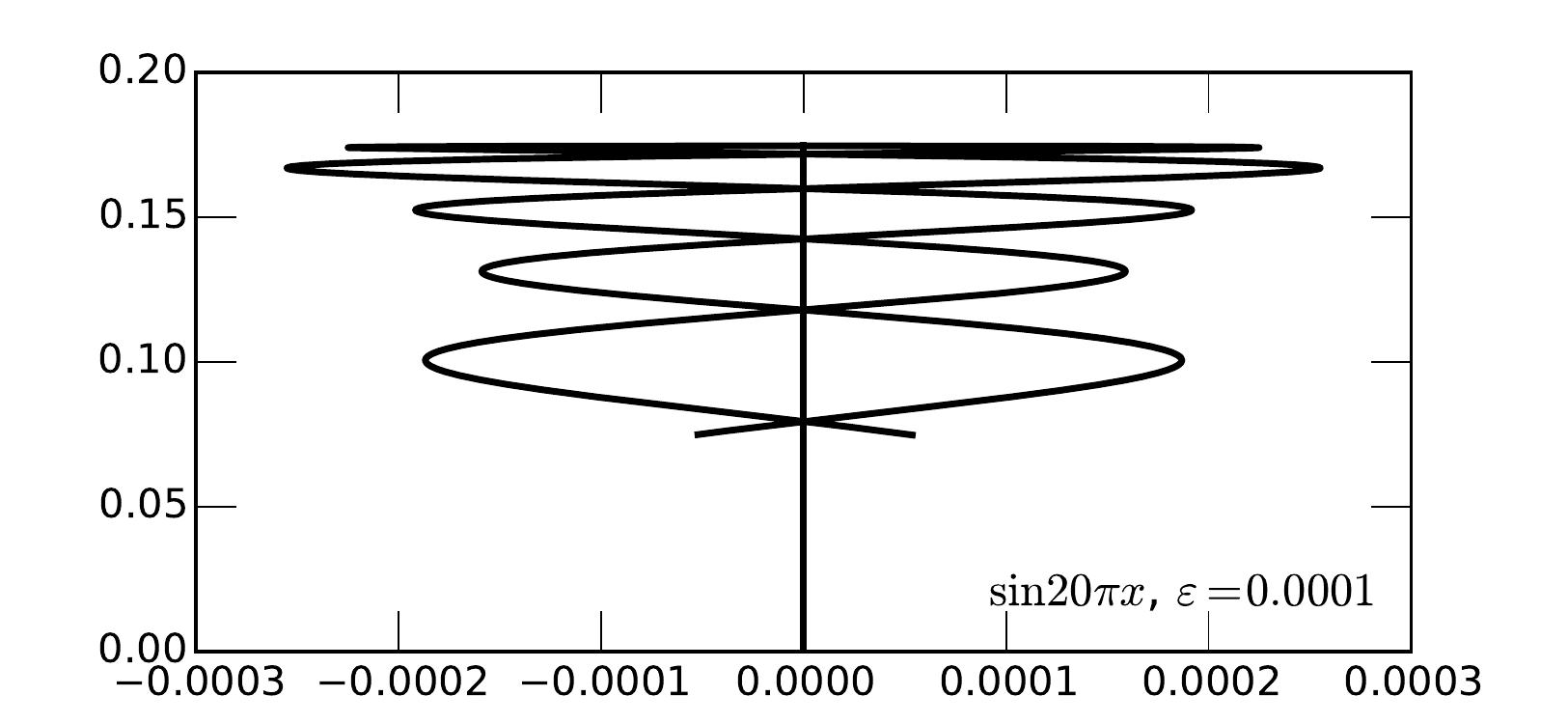}
   \caption{Infinitesimal $0$-bubbles for a perturbation of the Möbius map $f=\frac{z+0.5}{1+0.5z}$ by the map $g=\sin 2\pi nx$,  $n=1,2,3,4,5,10$. The pictures are rescaled horizontally. The vertical segment on each picture is the $0$-bubble for $f$.}\label{fig-bubbles1}
\end{center}
\end{figure}

In the general case, one can try to implement the construction described in \autoref{sec-expl-constr} as a computer program.
The author haven't done this yet.
For perturbations of Möbius maps, a simpler approach is described below.

Take a map $f + \eps g$ where $f$ is as in \autoref{prop-frac}, and $g$ is a trigonometric polynomial. \autoref{fig-bubbles1} shows infinitesimal $0$-bubbles of $f+\eps g$.
\begin{definition*}
 An \emph{infinitesimal 0-bubble} for a perturbation $f+\eps g$ of an analytic circle diffeomorphism $f$ is the image of the segment $I_{0}$ for $f$ under the map
 $\omega \mapsto \overline{\tau}(f+\omega) + \eps \cdot \frac{d}{d\eps}|_{\eps=0}\overline \tau (f+\eps g+ \omega)$, i.e. under the linear approximation   to the complex rotation number.
\end{definition*}
The choice of $\eps$ is shown on each picture in \autoref{fig-bubbles1}, but it does not essentially affect the shape of the infinitesimal bubble.
In the lower part of bubbles,  $\left.\frac{d}{d\eps}\right|_{\eps=0}\overline \tau (f+\eps g+ \omega)$ tends to infinity. So the linear approximation is not accurate, and this part of infinitesimal bubbles  is not shown on the picture.

The following proposition enables us to draw infinitesimal bubbles. Its proof follows the same scheme as the computation in \cite[Section 2.2.3]{Ris}; it is postponed till \autoref{sec-tau'}.
\begin{proposition}
\label{prop-tau'}
Let $f,g$ be as above. Let $\gamma$ be a curve in $\bbC/\bbZ$ which is close to $\bbR/\bbZ$, passes below the attractor and  above the repellor of $f+\omega$, $\omega\in I_0$. Then  
\begin{equation}
\label{eq-tau'}
\left. \frac{d}{d\eps}\right|_{\eps=0}\overline \tau (f+\eps g+ \omega)=\int_{\gamma} \frac{g(z)}{f'(z)} (H_{\omega}'(z))^2 dz
\end{equation}
where $H_{\omega}$ uniformizes $E(f+\omega)$. As in \autoref{prop-frac}, one can compute  $H_{\omega}$ explicitly. The derivatives in the right-hand side are with respect to $z$.
\end{proposition}

For any trigonometric polynomial $g$ (say, $g(x)=\sin 2\pi n x$), the change of variable $w = \pi(z)$ turns the integral \eqref{eq-tau'} into an integral of a rational function along the closed loop $\pi(\gamma)$. We then  compute it explicitly via Residue theorem; for $n\ge 3$, the formulas become cumbersome and we use a computer algebra system GiNaC \cite{GiNaC}, \cite{Vollinga_ginac:symbolic} to obtain them. The infinitesimal bubbles thus obtained are shown in \autoref{fig-bubbles1}.

In certain cases, intersections of infinitesimal $0$-bubbles for $f+\eps g$ mean that for small $\eps$, the $0$-bubbles of $f+\eps g$ intersect as well, see Remark \ref{rem-inf-bub} below.

\begin{question}
    Is it true that the map $\omega\mapsto \tau(f+\omega)$ is injective (so that the  bubbles belong to the boundary of the set $\{\tau(f+\omega), \omega \in \bbH/\bbZ\}$)?
\end{question}
No, see \cite[Corollary 16]{XB_NG}.

\begin{question}
    How large are the bubbles?
\end{question}

In \cite[Main Theorem]{XB_NG} the authors prove that the $p/q$-bubble (with coprime $p,q$) is within a disc of radius  $D_f/ (4\pi q^2)$ tangent to $\bbR/\bbZ$ at $p/q$, where $D_f$ is the distortion of $f$, $D_f = \int_{\bbR/\bbZ} \left|\frac{f''(x)}{f'(x)} \right| dx$.

\begin{question}
\label{q-main}
    Can the bubbles intersect or self-intersect?
\end{question}

Here are several results in this direction.
\begin{proposition}
 If an analytic circle diffeomorphism $f$ is sufficiently close to a rotation in $C^2$ metrics, its different bubbles do not intersect.
\end{proposition}
\begin{proof}
Suppose that the distortion of $f$ satisfies $D_f<2\pi$, which holds true if $f $ is $C^2$-close to a rotation. For each $p/q$, take the  disc of radius  $D_f/ (4\pi q^2) < \frac{1}{2q^2}$ tangent to $\bbR/\bbZ$ at $p/q$. It is easy to verify that these discs do not intersect for different $p/q$. The bubbles are within such discs, so they do not intersect as well.
\end{proof}

This proposition does not imply that the bubbles of $f$ are not self-intersecting.  This article contains an affirmative answer to \autoref{q-main}:

\begin{maintheorem}
\label{thm-main}
\begin{enumerate}
\item  There exists a circle diffeomorphism $f$ such that its $0$-bubble is self-intersecting. 
\item  For each rational $\frac pq$, there exists a circle diffeomorphism $f$ such that its $0$-bubble intersects its $\frac pq$-bubble.
\end{enumerate}
\end{maintheorem}

We do not assert  that these bubbles intersect transversely;  it is possible that they are tangent at a common point.

\begin{remark}
\label{rem-inf-bub}
Let $f=\frac{z+0.5}{1+0.5z}$ be the Möbius map that we chose to draw infinitesimal $0$-bubbles. Let $g = \sin 2\pi n x$, $n = 3,4,5,$ or $10$.
 Using the self-intersections of infinitesimal $0$-bubbles for $f+\eps g$, see \autoref{fig-bubbles1}, one may show that for sufficiently small $\eps$, the $0$-bubble of $f+\eps g$ is self-intersecting. This provides an alternative proof of the first part of Main Theorem. Here we sketch this proof.

Let $l_1(\eps)$ and $l_2(\eps)$ be two small intersecting arcs of the infinitesimal $0$-bubble for $f+\eps g$. Let $a_\eps, b_\eps$ and $c_\eps, d_\eps$ be the endpoints of $l_1(\eps), l_2(\eps)$ respectively. It is easy to verify that the lengths of the sides and the diagonals of the quadrilateral $a_\eps c_\eps b_\eps d_\eps$ are of order $\eps$, and $l_1(\eps)$, $l_2(\eps)$ are close to these diagonals. The $0$-bubble of $f+\eps g$ is $o(\eps)$-close to the infinitesimal $0$-bubble for $f+\eps g$, thus it contains a pair of curves that are $o(\eps)$-close to $l_1(\eps), l_2(\eps)$. This implies that the $0$-bubble of $f+\eps g$ is self-intersecting for small $\eps$.

\end{remark}

\section{Main Lemmas}
Part 1 of \nameref{thm-main} is based on \autoref{lem-invar} and the following lemma.
\begin{lemma}
\label{lem-oneFamily}
 For any hyperbolic analytic circle diffeomorphism $f_1$ with $\rot f_1=0$ and any analytic circle diffeomorphism $f_2\neq \id$, there exists an analytic diffeomorphism $f$ and $\omega \in \bbR/\bbZ \,  \setminus \, \{0\}$, such that $f$ and $f+\omega$ are analytically conjugate to $f_1,f_2$ respectively.
\end{lemma}
This lemma provides a non-restrictive sufficient condition for two analytic diffeomorphisms to appear (up to analytic conjugacies) in one and the same family of the form $f+\omega$.

Part 2 of \nameref{thm-main} also requires the following lemma, which is  interesting in its own right.
\begin{lemma}
\label{lem-realiz}
For any complex number $w\in \bbH/\bbZ$ and any natural number $m$, there exists a hyperbolic circle diffeomorphism $f$ having $2m$ fixed points and the complex rotation number $\bar \tau(f)=w$.
 \end{lemma}
 
\autoref{lem-invar} shows that complex rotation numbers can be used as invariants of analytic classification of families of circle diffeomorphisms; \autoref{lem-realiz} is a weak version of the realization of these invariants.
The following realization question is open: 

\vskip 0.1 cm

\begin{question}
 Which holomorphic self-maps of the upper half-plane are realized as $\omega \mapsto \tau(f+\omega)$ for some circle diffeomorphism $f$?
\end{question}

\section{Proof of the \nameref{thm-main} modulo Lemmas \ref{lem-oneFamily} and \ref{lem-realiz}}

\subsection{Part 1: Self-intersecting $0$-bubble}

This part of \nameref{thm-main} does not require \autoref{lem-realiz}. 

Fix a hyperbolic circle diffeomorphism $f_1$ with $\rot f_1=0$. Apply \autoref{lem-oneFamily} to $f_1$ and $f_2=f_1$. 

We get a circle diffeomorphism $f$ such that  $f, f+\omega$ with $\omega\neq 0 \mod 1$ are both analytically conjugate to $f_1$. Due to \autoref{lem-invar}, $\bar \tau(f) =\bar \tau(f_1)= \bar \tau(f+\omega)$. Note that $\bar \tau(f), \bar \tau(f+\omega)$ belong to the $0$-bubble for $f$ because $f, f+\omega$ have zero rotation number and are hyperbolic.

So the $0$-bubble for $f$ passes twice through the point $\bar \tau(f_{1})$. This completes the proof of \nameref{thm-main} (Part 1).
\begin{remark*}
 Using \autoref{lem-realiz}, one can also prove that the $0$-bubble may self-intersect at any prescribed point $w\in \bbH/\bbZ$. To achieve this, it is sufficient to start with $f_1$ provided by  \autoref{lem-realiz} such that $\bar \tau(f_1) =w$.
\end{remark*}

\subsection{Part 2: Intersection of $0$-bubble and \texorpdfstring{$\frac pq$}{(p/q)}-bubble}

Take a hyperbolic circle diffeomorphism $f_2$ with  $\rot f_2 = \frac pq$. Put $w: = \bar \tau(f_2)$.
Using \autoref{lem-realiz}, construct a hyperbolic circle diffeomorphism $f_1$ with zero rotation number such that $\bar \tau(f_1) = w$.

Now, two circle diffeomorphisms $f_1,f_2$ satisfy $\rot f_1 = 0, \rot f_2=\frac pq$ and $\bar \tau(f_1)=\bar \tau(f_2)$.

 \autoref{lem-oneFamily}   provides us with a circle diffeomorphism $f$ such that $f, f+\omega$ are conjugate to $f_1,f_2$. 
Due to \autoref{lem-invar}, 
$\bar \tau(f) = \bar \tau(f_1) =w$ and $\bar \tau(f+\omega) = \bar \tau(f_2) = w$. The point $w$ belongs to the $0$-bubble of $f$, because $\rot f = \rot f_1=0$ and $f$ is hyperbolic, and it also belongs to the $\frac pq$-bubble, because $\rot (f+\omega) =\rot (f_2) = \frac pq$ and $f+\omega$ is hyperbolic. Finally, the $0$-bubble and the $\frac pq$-bubble for $f$ intersect at $w$. This completes the proof of \nameref{thm-main} (Part 2).

\begin{remark*}
 In a similar way one can prove that the $0$-bubble and the $\frac pq$-bubble may intersect at any prescribed point $w\in \bbC/\bbZ$. This requires an analogue of \autoref{lem-realiz} for circle diffeomorphisms with non-zero rational rotation numbers; the proof of this analogue repeats the proof of \autoref{lem-realiz}, except  for some technical details. 
\end{remark*}

\section{Proof of \autoref{lem-oneFamily}}

We say that two circle diffeomorphisms $f_1,f_2$ have a \emph{Diophantine quotient} if  $\rot (f_1 f_2^{-1}) = : \omega$ is Diophantine. \autoref{lem-oneFamily} follows from two propositions below. 
\begin{proposition}
\label{prop-Dioph-dif}
 If two analytic circle diffeomorphisms $f_1,f_2$ have a Diophantine quotient and $\rot (f_1 f_2^{-1}) = : \omega$, then there exists an analytic diffeomorphism $f$ such that $f$ and $f+\omega$ are analytically conjugate to $f_1,f_2$ respectively.
\end{proposition}
\begin{proposition}
\label{prop-conjugate}
 Any hyperbolic analytic circle diffeomorphism $f_1$ with $\rot f_1=0$ is analytically conjugate to a diffeomorphism that has a Diophantine quotient with a given analytic circle diffeomorphism $f_2$,  $f_2\neq \id$.
\end{proposition}

\begin{proof}[Proof of \autoref{prop-Dioph-dif} ]
 Due to Herman -- Yoccoz Theorem (see \autoref{th-HermYoc}), in some analytic chart,  $f_1 f_2^{-1}$ is the rotation  by $\omega= \rot f_1f_2^{-1}$. 
Let $\tilde f_{1}, \tilde f_2$ be the diffeomorphisms $f_1,f_2$ in this analytic chart; then $\tilde f_1 \tilde f_2 ^{-1} (z) =z+\omega$. So 
$\tilde f_1  (z) =\tilde f_2(z) +\omega$, and we can take $f=\tilde f_2$. 
\end{proof}

\begin{proof}[Proof of \autoref{prop-conjugate}]
 
 Let $\mathcal A $ be the set of analytic diffeomorphisms of the form $\hat f_1 = h\circ f_1\circ h^{-1}$ for all possible analytic orientation-preserving diffeomorphisms~$h$. Then $\mathcal A$ is a linearly connected subset of the space of all analytic circle diffeomorphisms, because for each $h_1,h_2$, we can join $h_1$ to $h_2$ by a continuous family of analytic circle diffeomorphisms $h^t$. Now if we show that the continuous function $\hat f_1 \mapsto \rot (\hat f_1f_2^{-1})$ on $\mathcal A$  takes two distinct values, then  it takes all intermediate values, including Diophantine values. 
 
Let us  find two maps of the form $\hat f_1 = h\circ f_1\circ h^{-1}$ such that $\rot (\hat f_1f_2^{-1})$ attains  values $0$ and $1/2$: 
\begin{description}
  \item[$\mathbf{\rot (\hat f_1 f_2^{-1})=0}$] 
    Choose $h$ such that  for some point $a\in \bbR/\bbZ$, $\hat f_1(a) =f_2(a)$.
    This is possible, because $f_1\neq\id$ and $f_2\neq\id$.
    Then  $\hat f_1 f_2^{-1}(f_2(a))  =f_2(a)$, so $f_2(a)$ is a fixed point for $\hat f_1 f_2^{-1}$, and ${\rot (\hat f_1 f_2^{-1})=0}$.

  \item[$\mathbf{\rot (\hat f_1 f_2^{-1})=\frac 12}$]
    Choose two points $a,b\in \bbR/\bbZ$ such that these points and their preimages under $f_2$ are distinct and are ordered in the following way along the circle: $a,b, f_2^{-1}(a), f_2^{-1}(b)$. It is sufficient to  take $a$ not fixed and $b$ close to $a$.

    \begin{figure}[h]
      \centering
      \includegraphics[width=0.7\textwidth]{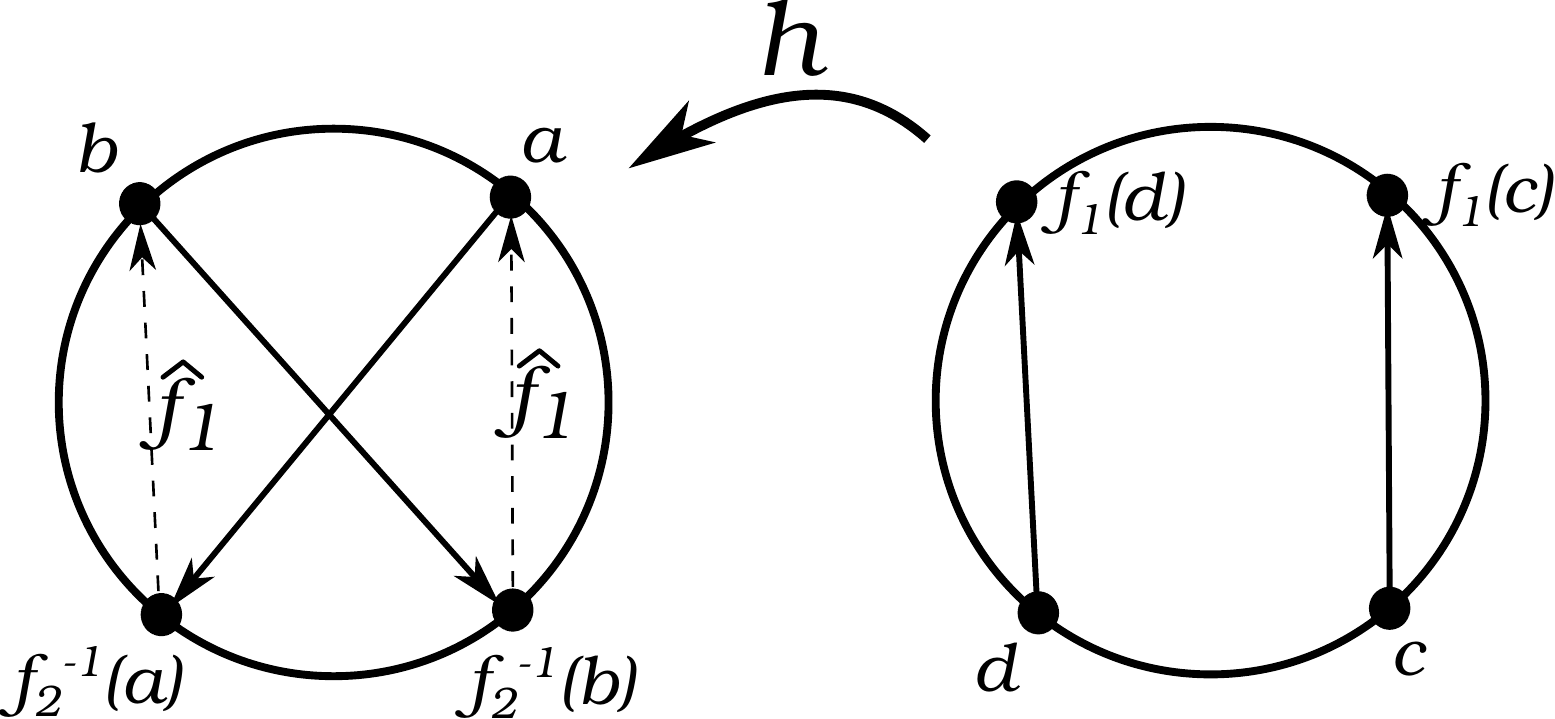}
      \caption{The choice of $h$ that yields  $\rot (\hat f_1 f_2^{-1})=\frac 12$}\label{fig-f1f2}
    \end{figure}

    Choose two points $c,d\in \bbR/\bbZ$ such that these points and their images under $f_1$ are distinct and  are ordered in the following way along the circle: $c,f_1(c), f_1(d), d$.  It is sufficient to  take $c$ and  $d$ near an attracting fixed point of $f_1$, on the different sides with respect to it.  

    Choose~$h$ that takes four points  $c,f_1(c), f_1(d), d$   to four points $f_2^{-1}(b),a,b, f_2^{-1}(a) $ (see \autoref{fig-f1f2}). Then $\hat f_1= h\circ f_1\circ h^{-1}$ satisfies $\hat f_1(f_2^{-1}(b))=a$, $\hat f_1(f_2^{-1}(a)) = b$, hence the point $a$ has period $2$ under $\hat f_1  f_2^{-1}$. So $\rot (\hat f_1 f_2^{-1})=\frac 12$. 
\end{description}

Finally, for some $h$, the maps $\hat f_1= h\circ f_1\circ h^{-1}$ and $f_2$ have a Diophantine quotient. 
\end{proof}
These two propositions imply \autoref{lem-oneFamily}.

The rest of the article is devoted to the proof of \autoref{lem-realiz}.

\section{Explicit construction of bubbles}
\label{sec-expl-constr}
\autoref{th-XB_NG} defines $\overline \tau(f+\omega)$, $\omega\in \bbR/\bbZ$, as a limit value of the map $\omega\to \tau(f+\omega)$ on the real axis. In this section, we describe $\overline \tau(f+\omega), \omega \in I_{0}$, as a modulus of an explicitly constructed complex torus $\mathcal E(f+\omega)$. 

This construction was proposed by X. Buff; see \cite{NG2012-en, XB_NG} for more details. The key idea of this construction is contained in \cite{Ris}, but there it was used in different circumstances.

\subsection{The complex torus \texorpdfstring{$\mathcal E(f)$}{ℰ(f)}}

Let $f$ be a hyperbolic diffeomorphism. Assume that  $\rot f=0$. 

Let $a_j, 1\le j\le 2m$, be its fixed points with multipliers $\lambda_j$. We suppose that $0<\lambda_{2j-1}<1<\lambda_{2j}$, i.e.\ even indices correspond to repellors, and odd indices correspond to attractors. Let $\psi_j \colon (\bbC, 0) \to (\bbC/\bbZ,a_j)$ be the corresponding linearization charts, i.e.\ $\psi^{-1}_j\circ f\circ \psi_j (z)=\lambda_j z$, $\psi_j(0)=a_j$,  $\psi_j(\bbR)\subset \bbR/\bbZ$, and $\psi_j$  preserve orientation on $\bbR$. We extend these charts by iterates of $f$ so that the image of $\psi_j$ contains $(a_{j-1}, a_{j+1})$.

\begin{figure}[h]
\begin{center}
\includegraphics[width=0.7\textwidth]{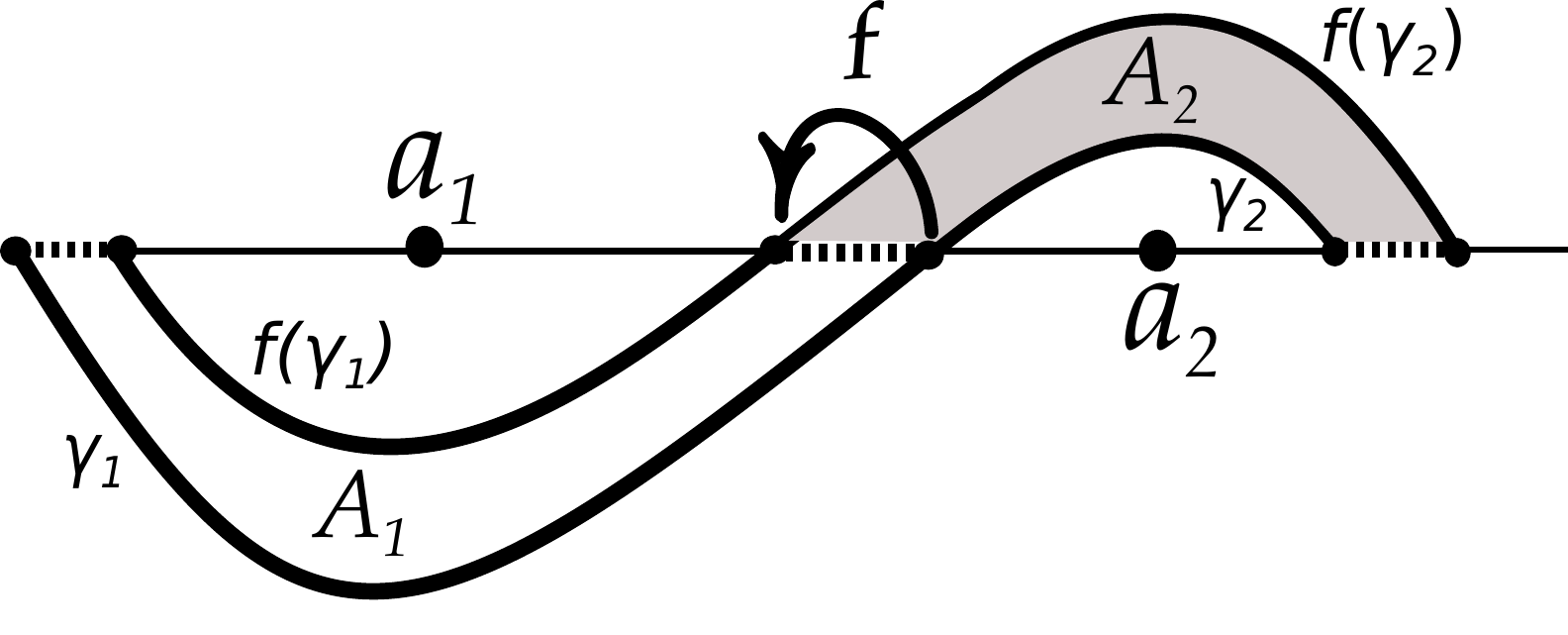}
\end{center}
 \caption{Construction of  $\mathcal E(f)$}\label{fig-Ef}
\end{figure}

Construct a simple loop $\gamma\subset \bbC/\bbZ$ (\emph{le courbe ascendente}, in terms of \cite{Ris}) such that $f(\gamma)$ is above $\gamma$ in $\bbC/\bbZ$. Namely, let $\gamma=\bigcup \gamma_j$; let $\gamma_j$ have its endpoints on $(a_{j-1}, a_j)$ and $(a_j, a_{j+1})$; let $\gamma_j$ be the image of an arc of a  circle under $\psi_j$; let $\gamma_{j}$ be above $\bbR/\bbZ$ if $j$ is even,  and below $\bbR/\bbZ$ if $j$ is odd.
Since $\psi_j$ conjugates $f$ to $z\mapsto \lambda_j z$, the curve $f(\gamma)$ is above $\gamma$ in $\bbC/\bbZ$.
 
 Let $\tilde \Pi \subset \bbC/\bbZ$ be a curvilinear cylinder between $\gamma$ and $f(\gamma)$ (see \autoref{fig-Ef}).
 Consider the complex torus $\mathcal E(f)$ being the quotient space of a neighborhood of $\tilde \Pi$ by the action of $f$. Due to the Uniformization Theorem, there exists $\tau\in \bbH/\bbZ$ and a biholomorphism $\tilde H_{\omega} \colon \mathcal E(f) \to \bbC / (\bbZ+\tau \bbZ)$ that takes $\gamma$ to a curve homotopic to   $\bbR/\bbZ$. Let $\tau (\mathcal E(f)):=\tau$ be the modulus of $\mathcal E(f)$.
 
For $\rot f = \frac pq$, the construction of $\gamma$ should be slightly modified: $\phi_j$ are linearizing charts of $f^q$ at its fixed points, $\gamma_j$ are arcs of circles in charts $\phi_j$, $\gamma=\bigcup \gamma_j$, $\gamma$ winds above repelling periodic points of $f$ and below attracting periodic points of $f$, and we choose $\gamma_j$ so that $f(\gamma)$ is above $\gamma$ in $\bbC/\bbZ$. The rest of the construction is analogous to the case of $\rot f=0$.

\begin{theorem}[\cite{NG2012-en}; see also {\cite[Sec. 6]{XB_NG}}]
\label{th-Buff-constr}
Let $f$ be a hyperbolic circle diffeomorphism with rational rotation number; define $\mathcal E(f)$ as  above. Then the modulus $\tau(\mathcal E(f))$ of the torus $\mathcal E(f)$ equals $\bar \tau(f)$.
\end{theorem}
Due to the construction, $\mathcal E(f)$ does not depend on the analytic chart on $\bbR/\bbZ$. This implies \autoref{lem-invar}.

So in order to prove \autoref{lem-realiz}, it is sufficient to find a circle diffeomorphism $f$ with $2m$ fixed points such that  $\tau(\mathcal E(f))=w$. 

\subsection{Cutting \texorpdfstring{$\mathcal E(f)$}{ℰ(f)} by the real line}
\label{subsec-cut}
Let $A_j\subset \tilde \Pi$ be the domain bounded by $\gamma_j, f(\gamma_j),$ and two segments of $\bbR/\bbZ$.
Note that the complex manifold $\tilde A_j:= A_j/f$ is an annulus, and $\tilde A_j \subset \tilde \Pi/f = \mathcal E(f) $. 

Let $\bbH^+=\bbH$ and $\bbH^-$ be the upper and the lower half-planes of $\bbC$ respectively.
From now on, we use the notation $A^{\pm}(\lambda)$ for the following standard annulus:
$ A^{\pm}(\lambda) :=  {\bbH}^{\pm}  / (z  \sim \lambda z).$ 
It is easy to see that its modulus is $\frac{\pi }{|\log \lambda|}$.

\begin{remark}
\label{rem-Ef-glue}
 The linearizing chart $\psi_j$ induces the map from $\tilde A_j$  to the standard annulus $A^{+}(\lambda_j) $ for even $j$,  and to $A^{-}(\lambda_j) $ for odd $j$. This follows from the fact that $\psi_j$ conjugates $f$ to $x \mapsto \lambda_j z$. 
 
 This gives a full description of $\mathcal E(f)$ in terms of multipliers and transition maps of $f$: 
 $\mathcal E(f)$ is biholomorphically equivalent to the quotient space of the annuli $ A^{\pm}(\lambda_j) $,  $\mod A^{\pm}(\lambda_j) =\frac{\pi }{|\log \lambda_j|}$, by the transition maps  $\psi_{j+1}^{-1}\circ\psi_j$ between linearizing charts of $f$.  
\end{remark}

\section{Circle diffeomorphisms with prescribed complex rotation numbers}
In this section, we prove \autoref{lem-realiz}.
\subsection{Scheme of the proof}

\autoref{rem-Ef-glue} above shows that $\mathcal E(f)$ can have any modulus, which  nearly implies \autoref{lem-realiz}. Indeed, we can obtain a complex torus of an arbitrary modulus by glueing some $2m$ annuli by some maps. We only need to show that there are no restrictions on possible multipliers and transition maps for an analytic  circle diffeomorphism. This follows from \autoref{th-moduli-classif} below.

The above arguments together with  \autoref{th-moduli-classif}  show that $\mathcal E(f)$ can be biholomorphic to a standard torus of any modulus; however we must also check that this biholomorphism matches the generators, as required by the definition of $\tau(\mathcal E(f))$, see \autoref{sec-expl-constr} above.
The formal proof of \autoref{lem-realiz}, with the explicit construction of $f$ and the examination of generators, is contained in \autoref{sec-ml-proof}.

\subsection{Moduli of analytic classification of hyperbolic circle diffeomorphisms}

The following theorem is an analytic version of a smooth classification of hyperbolic diffeomorphisms due to G. R. Belitskii, see \cite[Proposition 2]{Bel}. The proof is completely analogous, but we provide it for the sake of completeness.
\begin{theorem}
 \label{th-moduli-classif}
 Suppose that we are given a tuple of $2m$ real numbers $\lambda_j$ with $0<\lambda_{2j-1}<1<\lambda_{2j}$, and a tuple of analytic orientation-preserving diffeomorphisms $\psi_{j;j+1} \colon \bbR^{+} \to \bbR^{-}$ such that $\psi_{j;j+1} (\lambda_j z) =\lambda_{j+1} \psi_{j;j+1}(z)$.

Then there exists an analytic circle diffeomorphism $f$ such that it has $2m$ fixed points with multipliers $\lambda_j$, and $\psi_{j;j+1}$ are transition maps between their  linearization charts $\psi_j$: $\psi_{j;j+1} = \psi_{j+1}^{-1} \circ \psi_{j}$.
\end{theorem}

\begin{remark*}
 It is also true that such $f$ is unique up to analytic conjugacy, so the data above is the modulus of an analytic classification of hyperbolic circle diffeomorphisms. Given $f$, transition maps $\psi_{j;j+1}$ are uniquely defined up to the following equivalence: $$( \dots \psi_{j-1;j} \dots )\sim  (\dots, a_{j}\psi_{j-1;j} (z/a_{j-1}), \dots )$$  for some numbers $a_j>0$ (see \cite[Proposition 3]{Bel}).
\end{remark*}

\begin{proof}
 Take $2m$ copies of the real axis and glue the $j$-th to the $(j+1)$-th copy by the map $\psi_{j;j+1}\colon \bbR^{+} \to \bbR^{-}$. We get a one-dimensional  $C^{\omega}$-manifold homeomorphic to the circle $\bbR/\bbZ$. 
 It is well-known that such manifolds are  $C^{\omega}$-equivalent to $\bbR/\bbZ$. Thus there exists a tuple of $C^{\omega}$ charts $\psi_j \colon \bbR \to \bbR/\bbZ$ such that  $\psi_{j;j+1} = \psi_{j+1}^{-1} \circ \psi_{j}$. Due to the equality $\psi_{j;j+1} (\lambda_j z) =\lambda_{j+1} \psi_{j;j+1}(z)$, the maps $\psi_j (\lambda_j \psi_{j}^{-1} (z))$ glue into the well-defined  $C^{\omega}$ circle diffeomorphism $f$. 
 
Let $a_j = \psi_j(0)$. Note that $f(a_j) = \psi_j (\lambda_j \psi_{j}^{-1} (a_j))  = \psi_j (\lambda_j \cdot 0) = \psi_j(0) = a_j$, so these points are fixed poins of $f$.

On a segment $(a_{j-1}, a_{j+1})$, the map $\psi_j$ conjugates $f = \psi_j \circ \lambda_j \psi_{j}^{-1}$ to $z \mapsto \lambda_j z$, so $\psi_j$ is a linearizing chart of a fixed point $a_j$, and $\lambda_j$ is the multiplier of $f$ at $a_j$.
\end{proof}

\subsection{Proof of \autoref{lem-realiz}, see \autoref{fig-allExpl}}
\label{sec-ml-proof}
Recall that our aim is to construct a circle diffeomorphism $f$ with $2m$ hyperbolic fixed points and the complex rotation number $w$.  

Consider the standard elliptic curve $E_{w} = \bbC/ (\bbZ + w\bbZ)$; let $\bbR/\bbZ$ and $w\bbR/w\bbZ$ be its first and second generator respectively. Take arbitrary $2m$ disjoint simple real-analytic loops $\nu_j \subset E_{w}$ along the second generator. 
Let $\mathcal A_j\subset E_{w} $ be the annulus between $\nu_j$ and $\nu_{j+1}$. Let $\Delta_j = \mathcal A_j\cap \bbR/\bbZ$, then $\Delta_j$ joins boundaries of $\mathcal A_j$.

We are going to construct a circle diffeomorphism $f$ with $2m$ fixed points, and a biholomorphism $H \colon \mathcal E(f) \to E_w$ such that $H (\tilde A_j) = \mathcal A_j \subset E_{w}$, where $\tilde A_j$ are the annuli in $\mathcal E(f)$ bounded by intervals of $\bbR/\bbZ$ as in \autoref{subsec-cut}. This biholomorphism $H$ will take the class of $\gamma$ in  $\mathcal E(f)$ to the class of $\bbR/\bbZ=\bigcup \Delta_j$ in $E_w$.  This will prove that the modulus of $\mathcal E(f)$ equals $w$.  

\textbf{Uniformize $\mathcal A_j$}

For each annulus $\mathcal A_j$ where $j$ is even, take $\lambda_j>1$ such that there exists a biholomorphism  $\tilde \Psi_j \colon  A^{+}(\lambda_j)\to \mathcal A_j $. For each annulus $\mathcal A_j$, where $j$ is odd, take $\lambda_j<1$ such that there exists a biholomorphism  $\tilde \Psi_j \colon  A^{-}(\lambda_j)\to \mathcal A_j $.  Each map $\tilde \Psi_j$ extends analytically to a neighborhood of $A^{\pm}(\lambda_j)$ in $\bbC^*/(z\sim \lambda_j z)$, because the boundaries of $\mathcal A_j$ are real-analytic curves $\nu_j$.
Assume that $\tilde \Psi_j^{-1} (\nu_j)$ is the left boundary of $A^{\pm}(\lambda_j)$, $\tilde \Psi_j^{-1} (\nu_j) = \bbR^-/(z\sim \lambda_j z)$; then $\tilde \Psi_j^{-1} (\nu_{j+1}) =\bbR^+/(z\sim \lambda_j z)$. 

Let $\Psi_j \colon \overline {\mathbb H}^{\pm} \smallsetminus \{0\} \to \mathcal A_j$ be the lift of $\tilde \Psi_j$ to the universal cover of $A^{\pm}(\lambda_j)$; then $\Psi_j (\lambda_j z) = \Psi_j (z)$. 
For each $j$, choose one of the preimages $\delta_j = \Psi_j^{-1}(\Delta_j)$. Let $l_j \in \bbR^-, r_j\in \bbR^+$ be the left and the right endpoint of $\delta_j$ respectively.
Consider the maps $ \psi_{j;j+1} \colon \bbR^{+}  \to \bbR^{-} $,
$$
\psi_{j;j+1}= \Psi_{j+1}^{-1} \circ \Psi_{j},
$$
where we choose the branch of $\Psi_{j+1}^{-1}$ so that $\psi_{j;j+1} (r_j) = l_{j+1}$. Note that $\psi_{j;j+1} (\lambda_j z) =\lambda_{j+1} \psi_{j;j+1}(z)$ because $\Psi_j (\lambda_j z) = \Psi_j (z)$.

\begin{figure}[t]
\begin{center}
\includegraphics[width=\textwidth]{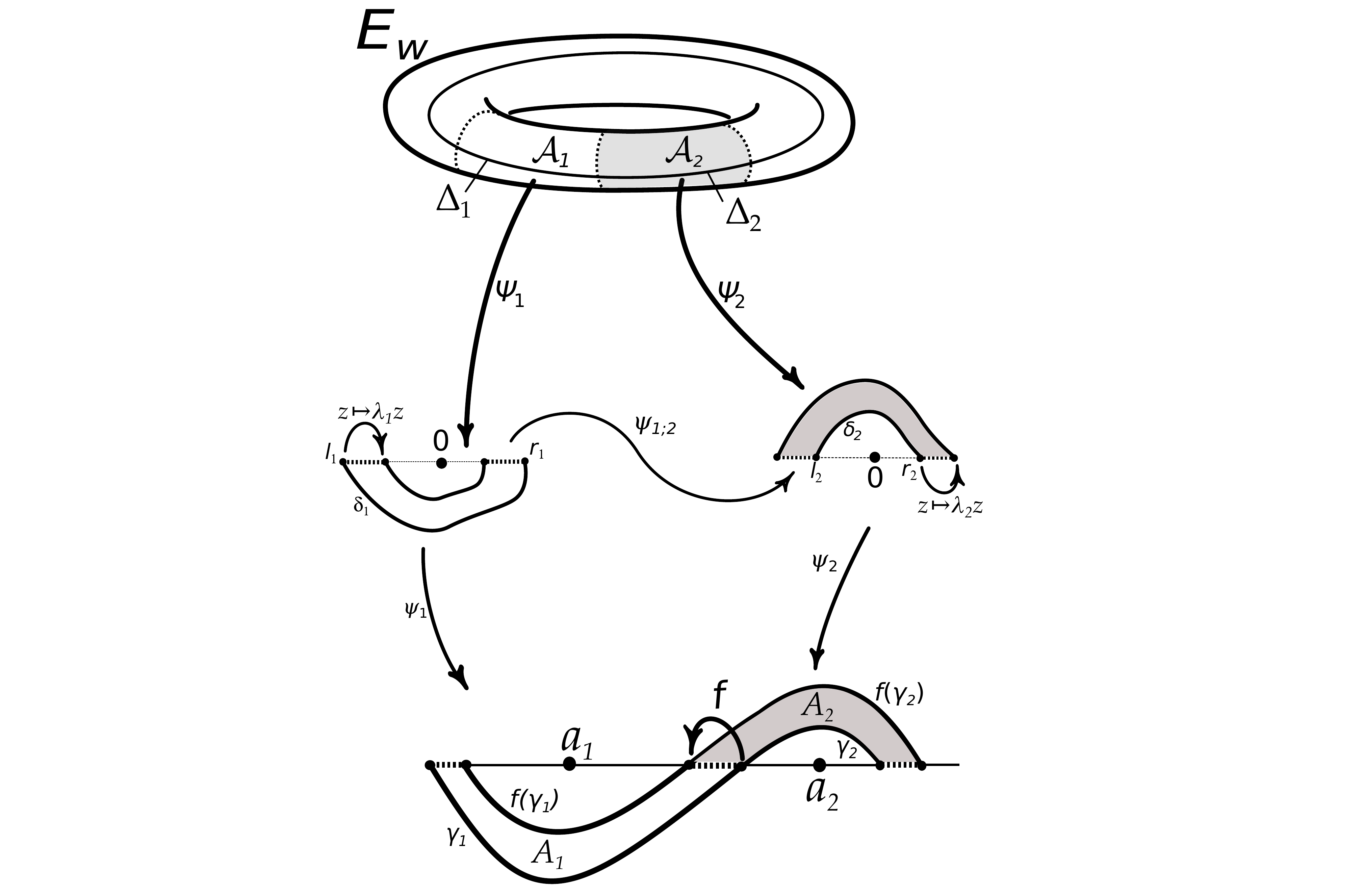}
\end{center}
 \caption{Proof of \autoref{lem-realiz}}\label{fig-allExpl}
\end{figure}

Now, the complex torus $E_w$ is biholomorphically equivalent  to the quotient space of annuli $A^{\pm}(\lambda_j)$ by the maps $\psi_{j;j+1}$. This, together with \autoref{rem-Ef-glue}, motivates the construction of $f$ below.

\textbf{Construct $f$ and a biholomorphism $H\colon \mathcal E(f) \to E_w$.}

 Use \autoref{th-moduli-classif} to construct $f$ with multipliers $\lambda_j$ and transition maps $\psi_{j;j+1}$. 
 
 Let $\psi_j$ be linearization charts of its fixed points; then $\psi_{j;j+1} = \psi_{j+1}^{-1}\circ \psi_j$. Let $\gamma,  \mathcal E(f), A_j, \tilde A_j$ be defined as in \autoref{sec-expl-constr} for this circle diffeomorphism $f$.
 
 Consider the tuple of maps $\Psi_j \circ \psi^{-1}_j$ on $A_j\subset \bbC/\bbZ$. These maps agree on the boundaries of $A_j$ due to the equality
 $$(\Psi_{j+1} \circ \psi^{-1}_{j+1})^{-1} \circ \Psi_j \circ \psi^{-1}_j = \psi_{j+1} \circ \Psi_{j+1}^{-1}\circ \Psi_j \circ \psi^{-1}_j = \psi_{j+1} \circ \psi_{j;j+1} \circ \psi^{-1}_j =\id,$$
 so they define one map on $\tilde \Pi$. They descend to the map $H \colon \mathcal E(f)\to E_w$ because $\psi_j$ conjugates $f$ to $z\mapsto \lambda_j z$ and $\Psi_j (\lambda_j z) = \Psi_j (z)$. Clearly, $H(\tilde A_j) =\mathcal A_j$.
 
 \textbf{$H$ takes the class of $\gamma$ in $\mathcal E(f)$ to the first generator of $E_w$.}
 
 Note that the curves $\psi_j(\delta_j)$ have common endpoints $\psi_j(r_j) = \psi_{j+1} (l_{j+1})$ since $\psi_{j;j+1} (r_j) = l_{j+1}$. So $\gamma':= \bigcup \psi_j(\delta_j)$ is a loop in $\bbC/\bbZ$ that passes above the attractors $ \psi_{2j-1}(0)$  and below the repellors $ \psi_{2j}(0)$ of $f$. So $\gamma'$ is homotopic to $\gamma$ in an annular neighborhood of $\bbR/\bbZ$ covered by linearizing charts of fixed points; the homotopy does not pass through fixed points. Hence $\gamma'$ is homotopic to $\gamma$ in $\mathcal E(f)$, i.e.\ corresponds to the first generator of $\mathcal E(f)$.

Finally, $H(\gamma') =\bigcup \Psi_j(\delta_j) = \bigcup \Delta_j = \bbR/\bbZ\subset E_w$.
This completes the proof of \autoref{lem-realiz}.

\appendix

\section{Derivatives of complex rotation number}
\label{sec-tau'}
In this section we compute $\frac{\partial }{\partial \omega}\overline{\tau}(f_{\omega})$ for a family of circle diffeomorphisms $f_{\omega}$. In particular, this yields \autoref{prop-tau'}. The computation is analogous to that of  \cite[Sec. 2.2.3]{Ris}.
 
Let $f_{\omega}$ be an analytic family of analytic circle diffeomorphisms. Let  $G_{\omega}:=\tilde H_{\omega}^{-1}$ where $\tilde H_{\omega}$ rectifies the complex torus $\mathcal E(f_\omega)$, see \autoref{sec-expl-constr}. Let $\tau(\omega) = \bar \tau(f_{\omega})$. Then 
\begin{equation}
\label{eq-conj}
f_{\omega} (G_{\omega}(z) ) = G_{\omega}(z+\tau(\omega)) \text{ for } z\in G_{\omega}^{-1}(\gamma).
\end{equation}
The Ahlfors--Bers theorem implies that the map $G_{\omega}$, if suitably normalized, depends analytically on $\omega$, see  \cite[Sec. 2.1, Proposition 2]{Ris}.  

Fix $\omega=\omega_0 \in \bbR/\bbZ$; in what follows, all derivatives with respect to $\omega $ are evaluated at $\omega=\omega_0$, and we will omit the lower indices in $f_{\omega}, G_{\omega}$ etc. Here and below $G'$, $f'$ are derivatives with respect to $z$; $G'_{\omega}, f'_{\omega}$, $\tau'_{\omega}$ are derivatives with respect to $\omega$. 

The following proposition clearly implies \autoref{prop-tau'}.

\begin{proposition}
\label{prop-deriv}
Let $f_{\omega}$, $G_{\omega}$ be as above. Then 
$$
\tau'_{\omega} =\int_{\gamma}  \frac{ f'_{\omega} (w)}{f'(w)} ((G^{-1})'(w))^2 dw
$$
where all derivatives are evaluated at $\omega=\omega_0$.
\end{proposition}

\begin{proof}
We may and will assume that the curve  $\gamma$ in the construction of $\mathcal E(f_{\omega})$ does not depend on $\omega$ in a small neighborhood of $\omega_0$.

Differentiate \eqref{eq-conj} with respect to $\omega$:
$$
f'_{\omega} |_{G(z)}
 + f'|_{G(z)} G'_{\omega}(z) = G'_{\omega} (z+\tau) + G' (z+\tau ) \tau'_{\omega}.
$$
Express $\tau'_{\omega}$ using this equation and the  identity $G'(z+\tau) = f'|_{G(z)} G'(z)$ (this is the derivative of \eqref{eq-conj}). We get 
$$
\tau'_{\omega} =\frac{f'_{\omega}|_{G(z)}}{G'(z+\tau)}+\frac{G'_{\omega}(z)}{G'(z)}   - \frac{G'_{\omega}(z+\tau)}{G'(z+\tau)}. 
$$
Integrate this expression along $G^{-1}(\gamma)$. The second and the third summands cancel out because the function $\frac{G'_{\omega} (z)}{ G'(z)}$ is holomorphic. We obtain 
$$
\tau'_{\omega}=\int_{G^{-1}(\gamma)} \frac{f'_{\omega}|_{G(z)} } {G'(z+\tau)} d z.
$$
Using again $G'(z+\tau)= G'(z) f'|_{G(z)}$ and making the change of variable $w=G(z)$, we get the desired formula.
\end{proof}
\printbibliography
\end{document}